\documentclass[11pt,a4paper]{amsart}
\usepackage{amsmath}
\usepackage{amssymb}
\usepackage{hyperref}

\begin{document}

\newcommand{\LL}{\mathbb{L}}
\newcommand{\RR}{\mathbb{R}^{2}}
\newcommand{\hh}[1]{\mathcal{H}_{\delta}^{#1}}
\newcommand{\HH}[1]{\mathcal{H}^{#1}}
\newcommand{\h}{\mathfrak{h}}
\newcommand{\g}{\mathfrak{g}}
\newcommand{\e}{\varepsilon}
\newcommand{\N}{\mathbb{N}}
\newcommand{\R}{\mathbb{R}}
\newcommand{\s}{\mathbb{S}}
\newcommand{\sub}{\subseteq}
\newcommand{\diam}{\text{diam}}
\newcommand{\Z}{\mathbb{Z}}
\newcommand{\Q}{\mathbb{Q}}
\newcommand{\f}{\mathfrak{f}}
\renewcommand{\H}{\mathbb{H}}
\newcommand{\n}{\mathfrak{n}}
\newcommand{\m}{\mathfrak{m}}

\date{\today}

\newtheorem{theorem}{Theorem}[section]
\newtheorem{lemma}[theorem]{Lemma}
\newtheorem{corollary}[theorem]{Corollary}
\newtheorem{proposition}[theorem]{Proposition}
\newtheorem{claim}[theorem]{Claim}

\theoremstyle{definition}
\newtheorem{definition}[theorem]{Definition}
\newtheorem{example}[theorem]{Example}
\newtheorem{remark}[theorem]{Remark}

\title{Small Furstenberg sets}

\author{Ursula Molter and Ezequiel Rela}

\address{Departamento de
Matem\'atica \\ Facultad de Ciencias Exactas y Naturales\\
Universidad de Buenos Aires\\ Ciudad Universitaria, Pabell\'on
I\\ 1428 Capital Federal\\ ARGENTINA\\ and IMAS - CONICET, Argentina}
 \email[Ursula Molter]{umolter@dm.uba.ar}
\email[Ezequiel Rela]{erela@dm.uba.ar} 

\thanks{This research  is partially supported by
Grants: PICT2006-00177, UBACyT X149 and CONICET PIP368}

\begin{abstract}
For $\alpha$ in $(0,1]$, a subset $E$ of $\RR$ is called
\textit{Furstenberg set} of type $\alpha$ or $F_\alpha$-set if
for each direction $e$ in the unit circle there is a line segment
$\ell_e$ in the direction of $e$ such that the Hausdorff
dimension of the set $E\cap\ell_e$ is greater than or equal to
$\alpha$. In this paper we use generalized Hausdorff measures 
to give estimates on the size of these sets.
Our main result is to obtain a sharp dimension estimate for a whole class of {\em zero}-dimensional Furstenberg type sets. Namely, for
$\mathfrak{h}_\gamma(x)=\log^{-\gamma}(\frac{1}{x})$, $\gamma>0$, we
construct a set $E_\gamma\in F_{\mathfrak{h}_\gamma}$ of Hausdorff
dimension not greater than $\frac{1}{2}$.  Since in a previous work we showed that    $\frac{1}{2}$ is a lower bound  for the
Hausdorff dimension of any $E\in F_{\mathfrak{h}_\gamma}$, with
the present construction, the value $\frac{1}{2}$ is sharp for
the whole class of Furstenberg sets associated to the zero
dimensional functions $\mathfrak{h}_\gamma$.
\end{abstract}

\maketitle

\section{Introduction}\label{sec:intro}
We study dimension properties of sets of Furstenberg type. In
particular we are interested in being able to construct very small
Furstenberg sets in a given class. We begin with
the definition of classical Furstenberg sets. 

\begin{definition}
For $\alpha$ in $(0,1]$, a subset $E$ of $\RR$ is called
\textit{Furstenberg set} of type $\alpha$ or $F_\alpha$-set if
for each direction $e$ in the unit circle there is a line segment
$\ell_e$ in the direction of $e$ such that the Hausdorff
dimension ($\dim_H$) of the set $E\cap\ell_e$ is greater than or equal to
$\alpha$. We will also say that such set $E$ belongs to the class
$F_\alpha$.
\end{definition}

It is known (\cite{wol99b}, \cite{kt01}) that $\dim_H(E)\ge\max\{2\alpha,\alpha+\frac{1}{2}\}$
for any $F_\alpha$-set $E\subseteq \RR$ and there are examples of
$F_\alpha$-sets $E$ with $\dim_H(E)\le
\frac{1}{2}+\frac{3}{2}\alpha$. Hence, if we denote by

\begin{equation*}
  \Phi(F_\alpha)=\inf \{\dim_H(E): E\in F_\alpha\},
\end{equation*}
then
\begin{equation}
\label{eq:dim}\max\left\{\alpha+\frac{1}{2} ; 2\alpha\right\}\le
\Phi(F_\alpha)\le\frac{1}{2}+\frac{3}{2}\alpha,\qquad
\alpha\in(0,1].
\end{equation}
In \cite{mr10} the left hand side of this inequality has been extended to the case of more general
dimension functions, i.e., functions that are not necessarily power functions
(\cite{Hau18}).

\subsection{Dimension Functions and Hausdorff measures}

\begin{definition}\label{def:dimfunc}
A function $h$ will be called \textit{dimension function} if it
belongs to the  class $\mathbb{H}$ which is defined as follows:
\begin{equation*}
 \mathbb{H}:=\{h:[0,\infty)\to[0:\infty),
 \text{non-decreasing, right continuous, }
h(0)=0 \}.
\end{equation*}
The important subclass of those $h\in\mathbb{H}$ that satisfy a doubling condition will be denoted by $\mathbb{H}_d$:
\begin{equation}
	\mathbb{H}_d:=\left\{h\in\mathbb{H}: h(2x)\le C h(x) \text{ for some }C>0\right\}.\nonumber
\end{equation} 
Let $g,h$ be two dimension functions. We will say that $g$ is
dimensionally smaller than $h$ and write $g\prec h $ if and only
if
    \begin{equation*}
    \lim_{x\to 0^+}\dfrac{h(x)}{g(x)}=0.
    \end{equation*}

A function $h\in\H$ will be called a ``zero dimensional dimension
function'' if $h\prec x^\alpha$ for any $\alpha>0$. We denote by
$\mathbb{H}_0$ the subclass of those functions.

As usual, the $h$-dimensional (outer) Hausdorff  measure $\HH{h}$
will be defined as follows. For a set $E\subseteq\R^n$ and
$\delta>0$, write
\begin{equation*}
    \hh{h}(E)=\inf\left\{\sum_i h(\diam(E_i)):E\subset\bigcup_i^\infty E_i,  \diam(E_i)<\delta \right\}.
\end{equation*}
The $h$-dimensional Hausdorff measure $\HH{h}$ of $E$ is defined
by
\begin{equation*}
    \HH{h}(E)=\sup_{\delta>0 }\hh{h}(E). \label{eq:h-meas}
\end{equation*}
\end{definition}

Recall that the Hausdorff dimension of a set $E\sub\R^n$ is the
unique real number $s$ characterized by the following properties:
\begin{itemize}
 \item $\HH{r}(E)=+\infty$ for all $r<s$.
 \item $\HH{t}(E)=0$ for all $s<t$.
\end{itemize}
Therefore, to prove that a given set   $E$ has dimension $s$, it is
enough to  check the preceding two properties, independently if
$\HH{s}(E)$ is zero, finite and positive, or infinite. However,
in general it is not true that, given a set $E$, there is a
function $h\in\H$, such that if $g\succ h$ then
$\HH{g}(E)=0$, and if $g\prec h$, then $\HH{g}(E)=+\infty$.
The difficulties arise from two results due to Besicovitch  (see
\cite{rog98} and references therein). The first says that if a set
$E$ has null $\HH{h}$-measure for some $h\in\H$, then there
exists a function  $g$ which is dimensionally smaller than $h$ and for which still $\HH{g}(E)=0$.
Symmetrically, the second says that if a \emph{compact} set $E$
has non-$\sigma$-finite $\HH{h}$ measure, then there exists a
function $g\succ h$ such that $E$ has also non-$\sigma$-finite
$\HH{g}$ measure. These two results imply that if a set $E$
satisfies that there exists a function  $h$ such that
$\HH{g}(E)>0$ for any $g\prec h$ and $\HH{g}(E)=0$ for any
$g\succ h$, then it must be the case that
$0<\HH{h}(E)$ and $E$ has $\sigma$-finite $\HH{h}$-measure. 

Now consider the set $\LL$ of Liouville numbers (see \cite{ek06}). It is shown in \cite{or06} that there are two proper nonempty subsets $\LL_0, \LL_\infty\sub\H$  of dimension functions, such that $\HH{h}(\LL)=0$ for all $h\in\LL_0$ and for all $h\in\LL_\infty$, the set $\LL$ has non $\sigma$-finite $\HH{h}$-measure. Therefore, this example shows that in the general setting of dimension functions, in order  to detect the precise size of a given set, one needs to estimate a dimensional gap starting from a conjectured dimension function, both from above and from below. Precisely, one wants to measure how fast the quotient  (or gap) $\frac{h}{g}$ between $g, h\in\H$ grows, whenever $g\prec h$ goes to zero.

\subsection{Generalized Furstenberg sets. Statement of the main result}

The analogous definition of Furstenberg sets in the setting of
dimension functions is the following.
\begin{definition}\label{def:furs}
Let $\h$ be a dimension function. A set $E\subseteq\RR$ is a
Furstenberg set of type $\h$, or an $F_\h$-set, if for each
direction $e\in\s$ there is a line segment $\ell_e$ in the
direction of $e$ such that  $\HH{\h}(\ell_e \cap E)>0$.
\end{definition}

In \cite{mr10} we proved that the appropriate dimension function
for an $F_\h$ set $E$ must be dimensionally not much smaller than $\h^2$
and $\h\sqrt{\cdot}$ (this is the generalized version of the left
hand side of \eqref{eq:dim}), the latter with some additional
conditions on $\h$. 
In particular, for the zero-dimensional Furstenberg-type sets $E$ belonging to $F_{\h_\gamma}$, where
$\h_\gamma\in\H_0$ (see Definition \ref{def:dimfunc}) is defined by
$\h_\gamma(x)=\frac{1}{\log^\gamma(\frac{1}{x})}$, we showed that $\dim_H(E)\ge\frac{1}{2}$.

In the present work we look at a refinement of the upper bound for the dimension of Furstenberg sets. Since we are looking for upper
bounds on a class of Furstenberg sets, the aim will be to explicitly construct  a very small set belonging to the given class.

We first consider the classical case of power functions,  $x^\alpha$, for $\alpha >0$. Recall that for this case, the known upper bound implies that, for any positive
$\alpha$, there is a set $E\in F_\alpha$ such that
$\HH{\frac{1+3\alpha}{2}+\e}(E)=0$ for any $\e>0$. By looking closer at  Wolff's arguments, it can be seen that in fact it is true that $\HH{g}(E)=0$ for any dimension function $g$ of the form
\begin{equation}\label{eq:supertrivial}
g(x)=x^\frac{1+3\alpha}{2}\log^{-\theta}\left(\frac{1}{x}\right),\qquad \theta>\frac{3(1+3\alpha)}{2}+1.
\end{equation}
Further, that argument can be modified (Theorem \ref{thm:logFalpha}) to sharpen on the logarithmic gap, and therefore improving \eqref{eq:supertrivial} by proving the same result for any $g$ of the form
\begin{equation}\label{eq:mediumtrivial}
g(x)=x^\frac{1+3\alpha}{2}\log^{-\theta}\left(\frac{1}{x}\right),\qquad \theta>\frac{1+3\alpha}{2}.
\end{equation}

However, this modification will not be sufficient for our main objective, which is to reach the zero dimensional case. More precisely, we will focus at the endpoint $\alpha=0$, and give a
complete answer about the \emph{exact dimension} of  a class of Furstenberg sets. We will prove in Theorem \ref{thm:sqrth3/2} that, for any given $\gamma>0$, there exists a set
$E_\gamma\sub\R^2$ such that 
\begin{equation}\label{eq:smallmain}
 E_\gamma\in F_{\h_\gamma} \text{ for }\h_\gamma(x)=\frac{1}{\log^\gamma(\frac{1}{x})} \text{ and } \dim_H(E_\gamma)\le\frac{1}{2}.
\end{equation}
This result, together with the results from \cite{mr10} mentioned above, shows that $\frac{1}{2}$ is sharp for the class $F_{\h_\gamma}$. In fact, for this family both  inequalities in \eqref{eq:dim} are in fact the equality $\Phi(F_{\h_\gamma})=\frac{1}{2}$.

In order to be able to obtain \eqref{eq:smallmain}, it is not enough to simply ``refine'' the construction of Wolff. He achieves the desired set by choosing a  specific set as the fiber in each direction. This set is known to have the correct dimension. To be able to reach the zero dimensional case, we need to handle the delicate issue of choosing an analogue zero dimensional set on each fiber. The main difficulty lies in being able to handle \emph{simultaneously} Wolff's construction and the proof of the fact that the fiber satisfies the stronger condition of having positive measure for the correct dimension function.

The paper is organized as follows.  In Section \ref{sec:furs}, for a given $\alpha > 0$, we carefully  develop the main construction of small Furstenberg ${\alpha}$ sets, and obtain dimension estimates for the class $F_\alpha$, $\alpha>0$. In Section \ref{sec:gfurs} we show that we can modify the argument of the previous section to include the zero-dimensional functions $\h_{\gamma}$ defined in \eqref{eq:smallmain}. The key ingredient is a lemma on Diophantine approximation regarding the size of zero dimensional fibers proved in  Section \ref{sec:fiber}. 
As usual, we will use the notation $A\lesssim B$ to indicate that
there is a constant $C>0$ such that $A\le C B$, where the
constant is independent of $A$ and $B$. By $A\sim B$ we mean that
both $A\lesssim B$ and $B\lesssim A$ hold.

\section{Upper Bounds for Furstenberg-type Sets}\label{sec:furs}
In this section we will concentrate on the right hand side of \eqref{eq:dim}.  This inequality has been proved by showing that there exists a set $E$
 in $F_{\alpha}$ such that  $\HH{s}(E)=0$
for any $s>\frac{1+3\alpha}{2}$.
However, by the result of Besicovitch cited in the Introduction, this does not necessarily imply that 
$\HH{h}(E)=0$ for any $h\succ x^\frac{1+3\alpha}{2}$.
Here we refine the arguments of Wolff to show, in Theorem \ref{thm:logFalpha}, that if $\theta > \frac{1+3\alpha}{2}$ there exists a set $E$ in $F_{\alpha}$ such that for $h_{\theta} := x^{\frac{1+3\alpha}{2}}\log^{-\theta}\left(\frac{1}{x}\right)$ we have that
$\HH{h_\theta}(E) = 0$.  

The purpose of this section is twofold. First,  we will carefully re-trace Wolff's arguments to show what is the key to be able to sharpen the logarithmic gap to obtain \eqref{eq:mediumtrivial} instead of \eqref{eq:supertrivial}. Second, and more importantly, we analyze the proof in detail to understand the nontrivial modification to be performed in Section \ref{sec:gfurs}.

We begin with a preliminary lemma about a very well distributed (mod 1) sequence.

\begin{lemma}\label{lem:discre}
 For $n\in \N$ and any real number $x\in [0,1]$, there is a pair of natural numbers $0\le j,k\le n-1$, such that
\begin{equation*}
\left|x-\left(\sqrt{2}\frac{k}{n}-\frac{j}{n}\right)\right|\le\frac{\log(n)}{n^2}.
\end{equation*}
\end{lemma}
This lemma is a consequence of Theorem 3.4 of \cite{kn74}, p125, in which an estimate is given about the discrepancy of the
fractional part of the sequence $\{n\alpha\}_{n\in\N}$ where $\alpha$ is an irrational of a certain type.

We also need to introduce the notion of $G$-sets, a common ingredient in the construction of Kakeya and Furstenberg sets.

\begin{definition}
 A $G$-set is a compact set $E\sub\R^2$ which is contained in the strip $\{(x,y)\in\R^2:0\le x \le 1\}$ such that for any $m\in[0,1]$ there is a line segment contained in $E$ connecting $x=0$ with $x=1$ of slope $m$, i.e.
\[
\forall m\in[0,1]\ \exists \ b\in\R: mx+b\in E,\ \forall \ x\in
[0,1].
\]
\end{definition}
Finally we need some notation for a thickened line.
\begin{definition}
 Given a line segment $\ell(x)=mx+b$, we define the $\delta$-tube associated to $\ell$ as
\[
S_\ell^\delta:=\{(x,y)\in\R^2:0\le x\le 1; |y-(mx+b)|\le\delta\}.
\]
\end{definition}\index{S@$S_\ell^\delta$, $\delta$-tube around the line $\ell$.}

Now we are ready to prove the main result of this section.

\begin{theorem}\label{thm:logFalpha}
For $\alpha\in(0,1]$ and $\theta>0$, define
$h_\theta(x)=x^{\frac{1+3\alpha}{2}}\log^{-\theta}(\frac{1}{x})$.
Then, if $\theta>\frac{1+3\alpha}{2}$, there exists a set
$E\in F_\alpha$ with $\HH{h_\theta}(E)=0$.
\end{theorem}

\begin{proof}
 
Fix $n\in\N$ and let $n_j$  be a sequence such that $n_{j+1}>n_j^j$. We consider $T$ to be the set defined as follows:
\begin{equation*}
T=\left\{x\in\left[\frac{1}{4},\frac{3}{4}\right]: \forall j \
\exists\ p,q\ ; q\le n_j^\alpha;
|x-\frac{p}{q}|<\frac{1}{n_j^2}\right\}.
\end{equation*}

It can be seen that $\dim_H(T)=\alpha$ (see Section \ref{sec:fiber}, Theorem \ref{thm:jarnikA}). 

If $\varphi(t)=\frac{1-t}{t\sqrt{2}}$ and $D=\varphi^{-1}\left([\frac{1}{4},\frac{3}{4}]\right)$, we have that $\varphi:D\to [\frac{1}{4},\frac{3}{4}]$ is bi-Lipschitz. Therefore the set
\begin{equation*}
T'=\left\{t\in\R:\frac{1-t}{t\sqrt{2}}\in
T\right\}=\varphi^{-1}(T)
\end{equation*}
also has Hausdorff dimension $\alpha$.

The main idea of our proof is to construct a set for which we have,
essentially, a copy of $T'$ in each direction and simultaneously
keep some optimal covering property.

Define, for each $n\in \N$,
\begin{equation*}
\Gamma_n:=\left\{\frac{p}{q}\in\left[\frac{1}{4},\frac{3}{4}\right],
q\le n^\alpha\right\}
\end{equation*}
and
\begin{equation*}
Q_n=\left\{ t:\frac{1-t}{\sqrt{2}t}=\frac{p}{q}\in
\Gamma_n\right\}=\varphi^{-1}(\Gamma_n).
\end{equation*}

To count the elements of $\Gamma_n$ (and $Q_n$), we take into
account that
\begin{equation*}
\sum_{j=1}^{\lfloor n^{\alpha}\rfloor}j\le \frac{1}{2} \lfloor
n^{\alpha}\rfloor(\lfloor n^{\alpha}\rfloor+1) \lesssim \lfloor
n^{\alpha}\rfloor^2\le n^{2\alpha}.
\end{equation*}
Therefore, $\#(Q_n)\lesssim n^{2\alpha}$. 

For $0\le j,k\le n-1$, define the line segments
\[
\ell_{jk}(x):=(1-x)\frac{j}{n}+x\sqrt{2}\frac{k}{n} \text{ for }
x\in[0,1],
\]
and their $\delta_n$-tubes $S_{\ell_{jk}}^{\delta_n}$ with
$\delta_n=\frac{\log(n)}{n^2}$. We will use during the proof the
notation $S_{jk}^n$ instead of $S_{\ell_{jk}}^{\delta_n}$. Also
define
\begin{equation}\label{eq:Gn}
G_n:=\bigcup_{jk}S_{jk}^n.
\end{equation}
Note that, by Lemma \ref{lem:discre}, all the $G_n$ are $G$-sets. 

For each $t\in Q_n$, we look at the points $\ell_{jk}(t)$, and 
define the set $S(t):=\{\ell_{jk}(t)\}_{j,k=1}^n$. Clearly, $\#(S(t))\le
n^2$. But if we note that, if $t\in Q_n$, then
\begin{equation*}
0\le\frac{\ell_{jk}(t)}{t\sqrt{2}}=\frac{1-t}{t\sqrt{2}}\frac{j}{n}+\frac{k}{n}=\frac{p}{q}\frac{j}{n}+\frac{k}{n}=\frac{pj+kq}{nq}<2,
\end{equation*}
we can bound $\#(S(t))$ by the number of non-negative rationals
smaller than 2 of denominator $qn$. Since $q\le n^\alpha$, we
have $\#(S(t))\le n^{1+\alpha}$. Considering \emph{all} the
elements of $Q_n$, we obtain $\#\left(\bigcup_{t\in
Q_n}S(t)\right)\lesssim n^{1+3\alpha}$. Let us define
\begin{equation*}
\Lambda_n:=\left\{(x,y)\in G_n: |x-t|\le \frac{\sqrt{2}}{n^2}\text{ for
some } t\in Q_n\right\}.
\end{equation*}

\begin{claim}
 For each $n$, take $\delta_n=\frac{\log(n)}{n^2}$. Then $\Lambda_n$ can be covered by $L_n$ balls of radius $\delta_n$ with $L_n \lesssim n^{1+3\alpha}$.
\end{claim}
To see this, it suffices to set a parallelogram on each point of
$S(t)$ for each $t$ in $Q_n$. The lengths of the sides of the
parallelogram are of order $n^{-2}$ and $\frac{\log(n)}{n^2}$, so
their diameter is bounded by  a constant time
$\frac{\log(n)}{n^2}$, which proves the claim.

{\em Back to the proof of the Theorem:} We   can now begin with the recursive construction that leads to the
desired set. The starting point $F_0$ will be essentially any $G$-set. We only add the assumption that it can be written as
\begin{equation*}
F_0=\bigcup_{i=1}^{M_0}S_{\ell_i^0}^{\delta^0},
\end{equation*}
(the union of $M_0$ line segments
$\ell^0_i=m^0_i+b^0_i$ with appropriate orientation and $\delta^0$-thickened). Each $F_j$
to be constructed will be a $G$-set of the form
\begin{equation*}
F_j:=\bigcup_{i=1}^{M_j}S_{\ell_i^j}^{\delta^j}, \qquad \text{
with } \ \ell^j_i=m^j_i+b^j_i.
\end{equation*}
 Having constructed $F_j$, consider the $M_j$ affine mappings
\begin{equation*}
A^j_i:[0,1]\times[-1,1]\rightarrow
S_{\ell^j_{i}}^{\delta^j}\qquad 1\le i\le M_j,
\end{equation*}
defined by
\begin{equation*}
  A^j_i\left(\begin{array}{c}
    x\\
    y
\end{array}\right)=\left(\begin{array}{cc}
    1 & 0\\
    m_i^j & \delta^j
\end{array}\right)\left(\begin{array}{c}
    x\\
    y
\end{array}\right)+\left(\begin{array}{c}
    0\\
    b^j_{i}
\end{array}\right).
\end{equation*}

Here is the key step: by the definition of  $T$, we can choose the sequence $n_{j}$ to grow as  fast  as we need  (this will not be the case in the next section). For example, we can choose $n_{j+1}$ large enough to satisfy
\begin{equation}\label{eq:njloglogMj}
\log\log(n_{j+1})>M_j
\end{equation}
and apply $A_i^j$ to the sets $G_{n_{j+1}}$ defined in
\eqref{eq:Gn} to obtain
\begin{equation*}
F_{j+1}=\bigcup_{i=1}^{M_j}A_i^j(G_{n_{j+1}}).
\end{equation*}
Since $G_{n_{j+1}}$ is a union of thickened line segments, we
have that
\begin{equation*}
F_{j+1}=\bigcup_{i=1}^{M_{j+1}}S_{\ell_i^{j+1}}^{\delta^{j+1}},
\end{equation*}
for an appropriate choice of $M_{j+1}$, $\delta_{j+1}$ and
$M_{j+1}$ line segments $\ell_i^{j+1}$. From the definition of
the mappings $A_i^j$ and since the set $G_{n_{j+1}}$ is a
$G$-set, we conclude that $F_{j+1}$ is also a $G$-set. Define
\begin{equation*}
E_j:=\{(x,y)\in F_j: x\in T'\}.
\end{equation*}
To cover $E_j$, we note that if $(x,y)\in E_j$, then $x\in T'$,
and therefore there exists a rational $\frac{p}{q}\in
\Gamma_{n_j}$ with
\begin{equation*}
\frac{1}{n_j^2}>\left|\frac{1-x}{x\sqrt{2}}-\frac{p}{q}\right|=|\varphi(x)-\varphi(r)|\ge\frac{|x-r|}{\sqrt{2}},\qquad
\text{ for some }\ r\in Q_{n_j}.
\end{equation*}
Therefore $(x,y)\in
\bigcup_{i=1}^{M_{j-1}}A^{j-1}_i(\Lambda_{n_j})$, so we conclude
that $E_j$ can be covered by $M_{j-1}n_{j}^{1+3\alpha}$ balls of
diameter at most $\frac{\log(n_{j})}{n_{j}^2}$. By our choice of
$n_j$, we obtain that $E_{j}$
admits a covering by $\log\log(n_{j})n_{j}^{1+3\alpha}$ balls of
the same diameter. Therefore, if we set $F=\bigcap_j F_j$ and
$E:=\{(x,y)\in F: x\in T'\} $ we obtain that
\begin{eqnarray*}
\HH{h_\theta}_{\delta_j}(E) & \lesssim  & n^{1+3\alpha}_{j}\left(\log\log(n_j)\right)h_\theta\left(\frac{\log(n_j)}{n^2_j}\right)\\
    & \lesssim & n^{1+3\alpha}_{j}\left(\log\log(n_j)\right)\left(\frac{\log(n_j)}{n^2_j}\right)^{\frac{1+3\alpha}{2}}\log^{-\theta}\left(\frac{n^2_j}{\log(n_j)}\right)\\
    & \lesssim & \left(\log\log(n_j)\right)\log(n_j)^{\frac{1+3\alpha}{2}-\theta}\lesssim  \log^{\frac{1+3\alpha}{2}+\e-\theta}(n_j)
\end{eqnarray*}
for large enough $j$. Therefore, for any
$\theta>\frac{1+3\alpha}{2}$, the last expression goes to zero.
In addition, $F$ is a $G$-set, so it must contain a line segment
in each direction $m\in[0,1]$. If $\ell$ is such a line segment,
then
\begin{equation*}
\dim_H(\ell\cap E)=\dim_H(T')\ge\alpha.
\end{equation*}
The final set of the proposition is obtained by taking eight
copies of $E$, rotated to achieve \emph{all} the directions in
$\s$.
\end{proof}

\section{Upper Bounds for small Furstenberg-type Sets}\label{sec:gfurs}

In this section we will focus on the class $F_\alpha$ at the
endpoint $\alpha=0$. Note that all preceding results involved only the case for which $\alpha > 0$. Introducing the generalized Hausdorff measures, we are able to handle an important class of Furstenberg type sets in $F_0$. 

The idea is to follow the proof of Theorem \ref{thm:logFalpha}. But in order to do that, we need to replace the set $T$ by a generalized version of it. A na\"ive approach would be to replace the $\alpha$ power in the definition of $T$ by a slower increasing function, like a logarithm. But in this case it is not clear that the set $T$ fulfills the condition of having positive measure for the corresponding dimension function (recall that we want to construct a set in $F_{\h_\gamma}$).  More precisely, we will need the following lemma.

\begin{lemma}\label{lem:fiber}
Let $r>1$ and consider the sequence $\n=\{n_j\}$ defined by $n_j=e^{\frac{1}{2}n_{j-1}^{\frac{4}{r}j}}$, the function 
$\f(x)=2\log(x^2)^\frac{r}{2}$ and the set
\begin{equation*}
T=\left\{x\in\left[\frac{1}{4},\frac{3}{4}\right]\setminus \Q: \forall j \ \exists\ p,q\ ; q\le
\f(n_j); |x-\frac{p}{q}|<\frac{1}{n_j^2}\right\}.
\end{equation*}
Then we have that $\HH{\h}(T)>0$ for $\h(x)=\frac{1}{\log(\frac{1}{x})}$.
\end{lemma}

\begin{remark}
We postpone the proof of this lemma to the next section, but we emphasize the following fact: in this  case, the construction of this new set $T$, does not allow us, as in \eqref{eq:njloglogMj}, to freely choose the sequence $n_j$. On one  hand we need the sequence to be quickly increasing to prove that the desired set is small enough, but not arbitrarily fast, since on the other hand, we need to impose some control to be able to prove that the fiber has the appropriate \emph{positive} measure. 
\end{remark}

 With this lemma, we are able to prove the main result of this section. We have the next theorem.

\begin{theorem}\label{thm:sqrth3/2}
Let $\h=\frac{1}{\log(\frac{1}{x})}$. There exists a set $E\in
F_\h$ such that $\dim_H(E)\le \frac{1}{2}$.
\end{theorem}

\begin{proof}
We will use essentially a copy of $T$ in each direction in the construction of the desired set to fulfill the conditions required to be an $F_\h$-set. Let $T$ be the set defined in Lemma \ref{lem:fiber}. Define $T'$ as
\begin{equation*}
T'=\left\{t\in\R:\frac{1-t}{t\sqrt{2}}\in
T\right\}=\varphi^{-1}(T),
\end{equation*}
where $\varphi$ is the same bi-Lipschitz function from the proof of Theorem \ref{thm:logFalpha}. Then $T'$ has positive $\HH{\h}$-measure. Let us define the
corresponding sets of Theorem \ref{thm:logFalpha} for this generalized case.
\begin{equation*}
\Gamma_n:=\left\{\frac{p}{q}\in\left[\frac{1}{4},\frac{3}{4}\right],
q\le \f(n)\right\},
\end{equation*}
\begin{equation*}
Q_n=\left\{ t:\frac{1-t}{\sqrt{2}t}=\frac{p}{q}\in
\Gamma_n\right\}=\varphi^{-1}(\Gamma_n).
\end{equation*}
Now the estimate is $\#(Q_n)\lesssim
\f^{2}(n)=\log^r(n^2)\sim\log^r(n)$, since
\begin{equation*}
\sum_{j=1}^{\lfloor \f(n)\rfloor}j\le \frac{1}{2} \lfloor
\f(n)\rfloor(\lfloor \f(n)\rfloor+1) \lesssim \lfloor
\f(n)\rfloor^2.
\end{equation*}
For each $t\in Q_n$, define $S(t):=\{\ell_{jk}(t)\}_{j,k=1}^n$.
If $t\in Q_n$, following the previous ideas, we obtain that
\begin{equation*}
\#(S(t))\lesssim n\log^\frac{r}{2}(n),
\end{equation*}
and therefore
\begin{equation*}
\#\left(\bigcup_{t\in Q_n}S(t)\right)\lesssim
n\log(n)^{\frac{3r}{2}}.
\end{equation*}
Now we estimate the size of a covering of
\begin{equation*}
\Lambda_n:=\left\{(x,y)\in G_n: |x-t|\le \frac{\sqrt{2}}{n^2}\text{ for
some } t\in Q_n\right\}.
\end{equation*}
For each $n$, take $\delta_n=\frac{\log(n)}{n^2}$. As before, the
set $\Lambda_n$ can be covered with  $L_n$ balls of radius
$\delta_n$ with $L_n\lesssim n\log(n)^\frac{3r}{2}$.

Once again, define $F_j$, $F$, $E_j$ and $E$ as before. Now the
sets $F_j$ can be covered by fewer than
$M_{j-1}n_j\log(n_j)^\frac{3r}{2}$ balls of diameter at most
$\frac{\log(n_{j})}{n_{j}^2}$. Now we can verify that, since each
$G_n$ consist of $n^2$ tubes, we have that $M_j=M_0n_1^2\cdots
n_j^2$.

Now we are again at the key point: by the choice of the sequence $\{n_j\}$ in the definition of $T$, we can also verify that the sequence  satisfies
the relation $\log{n_{j+1}}\ge M_{j}=M_0n_1^2\cdots n_j^2$, and
therefore we have the bound

\begin{equation*}
\dim_H(E)\le\underline{\dim}_B(E)\le\varliminf_j\frac{\log\left(\log(n_{j})
n_j\log(n_j)^\frac{3r}{2}
\right)}{\log\left(n_{j}^2\log^{-1}(n_j)\right)}=\frac{1}{2},
\end{equation*}
 where $\underline{\dim}_B$ stands for the lower box dimension.
Finally, for any $m\in[0,1]$ we have a line segment $\ell$ with
slope $m$ contained in $F$. It follows that $\HH{\h}(\ell\cap
E)=\HH{\h}(T')>0$.
\end{proof}

We remark that the argument in this particular result is
essentially the same needed to obtain the family of Furstenberg
sets $E_\gamma\in F_{\h_\gamma}$ for
$\h_\gamma(x)=\frac{1}{\log^\gamma(\frac{1}{x})}$,
$\gamma\in\R_+$, such that $\dim_H(E_\gamma)\le \frac{1}{2}$
announced in the Introduction.

\section{Proof of Lemma \ref{lem:fiber}}\label{sec:fiber}

The purpose of this section is to prove Lemma \ref{lem:fiber}. Our proof relies on a variation of a Jarn\'ik type theorem on Diophantine approximation.  We begin with some preliminary results on Cantor type constructions that will be needed.

\subsection{Cantor sets}\label{sec:cantor}

In this section we introduce the construction of sets of Cantor
type in the spirit of \cite{fal03}. By studying two quantities,
the number of children of a typical interval and some separation
property, we obtain sufficient conditions on these quantities
that imply the positivity of the $h$-dimensional measure for a
test function $h\in\H$.

We will need a preliminary elemental lemma about concave
functions. The proof is straightforward.
\begin{lemma}\label{lem:lemaconcava}
Let $h\in \H$ be a concave dimension function. Then
\[
 \min\{a,b\}\le \frac{a}{h(a)}h(b)\quad\text{ for any } a,b\in\R_+.
\]
\end{lemma}

\begin{proof}
Since $h$ is concave, we have $h\left(\frac{x+y}{2}\right) \leq \frac{h(x)+h(y)}{2}$. We consider two separate cases:
\begin{itemize}
    \item If $b\ge a$ then $\frac{a}{h(a)}h(b)\ge\frac{a}{h(a)}h(a)=a=\min\{a,b\}$.
    \item If $a> b$, then $\frac{a}{h(a)}h(b)\ge\frac{b}{h(b)}h(b)=b=\min\{a,b\}$ by concavity of $h$.    
\end{itemize}
\end{proof}
The following lemma is a natural extension of the ``Mass
Distribution Principle'' to the dimension function setting.
\begin{lemma}[$h$-dimensional mass distribution principle]\label{lem:massh}
 Let $E\sub\R^n$ be a set, $h\in\H$ and $\mu$ a probability measure on $E$. Let $\e>0$ and $c>0$ be positive constants such that for any $U\sub\R^n$ with $\emph{\diam}(U)<\e$ we have
\begin{equation*}
 \mu(U)\le c h(\emph{\diam}(U)).
\end{equation*}
Then $\HH{h}(E)>0$.
\end{lemma}
\begin{proof}
For any $\delta$-covering we have
\begin{equation*}
 0<\mu(E)\le \sum_i\mu(U)\le c\sum_i h(\diam(U)).
\end{equation*}
Then $\hh{h}>\frac{\mu(E)}{c}$ and therefore $\HH{h}(E)>0$.
\end{proof}
Now we present the construction of a Cantor-type set (see Example 4.6 in \cite{fal03}).
\begin{lemma}\label{lem:cantorgeneral}
Let $\{E_k\}$ be a decreasing sequence of closed subsets of the
unit interval. Set $E_0=[0,1]$ and suppose that the following
conditions are satisfied:
\begin{enumerate}
    \item Each $E_k$ is a finite union of closed intervals $I^k_j$.
    \item \label{eq:mk} Each level $k-1$ interval contains at least $m_k$ intervals of level $k$. We will refer to these as the ``children'' of an interval.
    \item The gaps between the intervals of level $k$ are at least of size $\e_k$, with $0<\e_{k+1}<\e_k$.
\end{enumerate}
Let $E=\bigcap_k E_k$. Define, for a concave dimension
function $h\in\H$, the quantity
\begin{equation*}
D^h_k:=m_1\cdot m_2\cdots m_{k-1}h(\e_km_k).
\end{equation*}
If $\varliminf_k D^h_k>0$, then $\HH{h}(E)>0$.
\end{lemma}

\begin{proof}
The idea is to use the version of the mass distribution principle
from Lemma \ref{lem:massh}. Clearly we can assume that the
property \eqref{eq:mk} of Lemma \ref{lem:cantorgeneral} holds
for \emph{exactly} $m_k$ intervals. So we can define a mass
distribution on $E$ assigning a mass of $\frac{1}{m_1\cdots m_k}$
to each of the  $m_1\cdots m_k$ intervals of level $k$. Now, for
any interval $U$ with $0<|U|<\e_1$, take $k$ such that
$\e_k<|U|<\e_{k-1}$. We will estimate the number of intervals of
level $k$ that could have non-empty intersection with $U$. For
that, we note the following:
\begin{itemize}
    \item $U$ intersects at most one $I^{k-1}_{j}$, since $|U|<\e_{k-1}$. Therefore it can intersect at most $m_k$ children of $I^{k-1}_{j}$.
    \item Suppose now that $U$ intersects $L$ intervals of level $k$. Then it must contain $(L-1)$ gaps of size at least $\e_k$. Therefore, $L-1\le\frac{|U|}{\e_k}$. Consequently $|U|$ intersects at most $\frac{|U|}{\e_k}+1\le 2\frac{|U|}{\e_k}$ intervals of level $k$.
\end{itemize}
From these two observations, we conclude that
\begin{equation*}
\mu(U)\le\frac{1}{m_1\cdots m_k}\min\left\{m_k,
\frac{2|U|}{\e_k}\right\}=\frac{1}{m_1\cdots
m_k\e_k}\min\{\e_km_k,2|U|\}.
\end{equation*}
Now, by the concavity of $h$, we obtain
\begin{equation*}
 \min\{\e_km_k,2|U|\}\le \frac{\e_km_k}{h(\e_km_k)}h(2|U|).
\end{equation*}
In addition (also by concavity), $h$ is doubling, so
$h(2|U|)\lesssim h(|U|)$ and then
\begin{equation*}
\mu(U)\lesssim\frac{\e_km_kh(|U|)}{m_1\cdots
m_k\e_kh(\e_km_k)}=\frac{h(|U|)}{m_1\cdots
m_{k-1}h(\e_km_k)}=\frac{h(|U|)}{D^h_k}.
\end{equation*}
Finally, if $\varliminf_k D_k^h>0$, there exists $k_0$ such
$\frac{1}{D_k^h}\le C$ for $k\ge k_0$ and we can use the mass
distribution principle with $C$ and $\e=\e_{k_0}$.
\end{proof}

\begin{remark}
In the particular case of $h(x)=x^s$, $s\in(0,1)$ we recover the
result of \cite{fal03}, where the parameter $s$ can be expressed
in terms of the sequences $m_k$ and $\e_k$. For the set
constructed in Lemma \ref{lem:cantorgeneral}, we have
\begin{equation}\label{eq:classicgeneral}
 \dim_H(E)\ge \frac{\log(m_1\cdots m_{k-1})}{-\log(m_k\e_k)}.
\end{equation}
\end{remark}

\subsection{Diophantine approximation - Jarn\'ik's Theorem}\label{sec:dioph}

The central problem in the theory of Diophantine approximation is, at its simplest level, to approximate irrational numbers by rationals. A classical theorem due to Jarn\'ik in this area is the following (see \cite{fal86}), which provides a result on the size of the set of real numbers that are well approximable. As usual, $d(x,\mathbb Z)$ denotes the distance from $x$ to the nearest integer.
\begin{theorem}\label{thm:jarnikB}
For $\beta\ge2$, define the following set:
\[
B_\beta=\left\{x\in[0,1]\setminus \Q:
d(xq,\mathbb Z)<\frac{1}{q^{\beta-1}}\text{ for infinitely many }
q\in\Z\right\}.
\]
Then $\dim_H(B_\beta)=\frac{2}{\beta}$.
\end{theorem}

In fact, there are several results (see \cite{khi24}, \cite{jar31}) regarding the more general problem of estimating the size of the set
\begin{equation}\label{eq:Bg}
B_\g:=\left\{x\in[0,1]\setminus \Q:
d(xq,\mathbb Z)\le\frac{q}{\g(q)}\text{ for infinitely many }
q\in\N\right\},
\end{equation} 
where $\g$ is any positive increasing function.

We can therefore see Lemma \ref{lem:fiber} as a result on the size of a set of well approximable numbers. We will derive the proof of Lemma \ref{lem:fiber} from the following proposition, where we prove a lower bound estimate for the set $B_\g$. The original idea is from \cite{fal03}, Theorem 10.3, where the result for the classical case in the form presented above in Theorem \ref{thm:jarnikB} is proved. Precisely, for $\h(x)=\frac{1}{\left(\g^{-1}(\frac{1}{x})\right)^2}$ we will find conditions on $\frac{\h}{h}$ with $h\prec \h$ to ensure that $\HH{h}(B_\g)>0$. To simplify the notation of the proof, let $\Delta(h,\h)(x)=\frac{\h(x)}{h(x)}$ and recall that $\H_d$ is the set of all dimension functions that satisfy a doubling condition.

\begin{proposition}\label{pro:Bglow}
Let $\g$ be a positive, increasing function satisfying 
\begin{equation}\label{eq:fastg} 
\g(x)\gg x^2 \quad  (x\gg 1)
\end{equation}
and 
\begin{equation}\label{eq:logg-1}
\g^{-1}(ab)\lesssim \g^{-1}(a)+\g^{-1}(b) \text{ for all }
a,b\ge1.
\end{equation}
Define $B_\g$ as in \eqref{eq:Bg} and let $h\in\H_d$ be a concave dimension function
such that $h\prec\h(x)=\frac{1}{\left(\g^{-1}(\frac{1}{x})\right)^2}$. Consider a sequence $\{n_k\}$ that satisfies:
\begin{itemize}
\item[(A)] $n_{k}\ge 3\g(2n_{k-1})$.
\item[(B)] $\log(n_k)\le \g(n_{k-1})$.
\end{itemize}
If
$\Delta(h,\h)(x)=\frac{\h(x)}{h(x)}=\frac{1}{h(x)\g^{-1}(\frac{1}{x})^2}$
satisfies
\begin{equation}\label{eq:critlow}
\varliminf_k\frac{1}{6^k\g^{2}(2n_{k-2})\Delta(h,\h)\left(\frac{1}{\log(n_{k})\g(2n_{k-1})}\right)}>0,
\end{equation}
then $\HH{h}(B_\g)>0$.
\end{proposition}
\begin{proof}
Define
\begin{equation}
 G_q:=\left\{x\in [0,1]\setminus \Q: d(xq,\mathbb Z)\le \frac{q}{\g(q)}\right\}.
\end{equation}
For each $q\in\N$, $G_q$ is the union of $q-1$ intervals (with no rational numbers) of
length $2\g(q)^{-1}$ and two more intervals of length
$\g(q)^{-1}$ at the endpoints of $[0,1]$.
Now define for each $q$
\[
G'_q:=G_q\cap \left(\frac{1}{\g(q)},1-\frac{1}{\g(q)}\right).
\]

Now, for each $n\in\N$ consider two prime numbers $p_1, p_2$ such
that $n\le p_1<p_2<2n$ (these can always be chosen for large $n$, see \eqref{eq:PrimeNumberTheorem}). We will prove that $G'_{p_1}$ and
$G'_{p_2}$ are disjoint and well separated. Note that if
$\frac{r_1}{p_1}$ and $\frac{r_2}{p_2}$ are centers of two of the
intervals belonging to $G'_{p_1}$ and $G'_{p_2}$, we have
\[
 \left|\frac{r_1}{p_1}-\frac{r_2}{p_2}\right|=\frac{1}{p_1p_2}|r_1p_2-r_2p_1|\ge\frac{1}{4n^2},
\]
since $r_1p_2-r_2p_1\neq 0$. Therefore, taking into account this
separation between the centers and the length of the intervals,
we conclude that for $x\in G'_{p_1}$ and $y\in G'_{p_2}$,
\[
 |x-y|\ge \frac{1}{4n^2}-\frac{2}{\g(n)}\ge \frac{1}{8n^2}\qquad(\text{ since }\g(n)\gg n^2).
\]
Let $\mathcal{P}_m^n$ be the set of all the prime numbers between
$m$ and $n$ and define
\[
    H_n:=\bigcup_{p\in\mathcal{P}_n^{2n}}G'_p.
\]
Then $H_n$ is the union of intervals of length at least
$\frac{2}{\g(2n)}$ that are separated by a distance of at least
$\frac{1}{8n^2}$.

Now we observe the following: If $n\in\mathbb N$, $n\le p\le 2n$ and $I$ is an interval with
$|I|>\frac{3}{n}$, then at least $\frac{p|I|}{3}$ of the
intervals of $G'_p$ are completely contained on $I$. To verify
this last statement, cut $I$ into three consecutive and congruent
subintervals. Then, in the middle interval there are at least
$\frac{p|I|}{3}$ points of the form $\frac{m}{p}$. All the
intervals of $G'_p$ centered at these points are completely
contained in $I$, since the length of each interval of $G'_p$ is
$\frac{2}{\g(p)}<\frac{|I|}{3}$.

In addition, by the Prime Number Theorem, we know that
$\#(\mathcal{P}_1^n)\sim\frac{n}{\log(n)}$, so we can find $n_0$
such that
\begin{equation}\label{eq:PrimeNumberTheorem}
\#(\mathcal{P}_n^{2n}) \ge \frac{n}{2\log(n)} \text{ for } n\ge
n_0.
\end{equation}

Hence, if $I$ is an interval with $|I|>\frac{3}{n}$, then there
are at least
\[
\frac{p|I|}{3}\frac{n}{2\log(n)}>\frac{n^2|I|}{6\log(n)}
\]
intervals of $H_n$ contained on $I$. Now we will construct a
Cantor-type subset $E$ of $B_\g$ and apply
Lemma \ref{lem:cantorgeneral}.

Consider the sequence $\{n_k\}$ of the hypothesis of the
proposition and let $E_0=[0,1]$. Define $E_{k}$ as the union of
all the intervals of $H_{n_{k}}$ contained in $E_{k-1}$. Then
$E_k$ is built up of intervals of length at least
$\frac{1}{\g(2n_k)}$ and separated by at least
$\e_k=\frac{1}{8n_k^2}$. Moreover, since
$\frac{1}{\g(2n_{k-1})}\ge \frac{3}{n_k}$, each interval of
$E_{k-1}$ contains at least
\[
m_k:=\frac{n_k^2}{6\log(n_k)\g(2n_{k-1})}
\]
intervals of $E_k$.

Now we can apply Lemma \ref{lem:cantorgeneral} to $E=\bigcap E_k$. Recall that $\e_k$ denotes the separation between $k$-level intervals and $m_k$ denotes the number of children of each of them. Consider $h\prec\h$, $h\in \H_d$, and recall the notation $\Delta(h,\h)(x)=\frac{\h(x)}{h(x)}$. Then
\begin{eqnarray*}
 D^h_k & = & m_1\cdot m_2\cdots m_{k-1}h(\e_km_k)\\
    & = & \frac{6^{-(k-2)}n^2_2\cdots n^2_{k-1}}{\log(n_2)\cdots\log(n_{k-1})\g(2n_1)\cdots\g(2n_{k-2})}h\left(\frac{1}{48\log(n_k)\g(2n_{k-1})}\right).\\
\end{eqnarray*}
Now we note that $n_k\ge \log(n_k)$ and, by hypothesis $(A)$, we
also have that  $n_k\ge\g(2n_{k-1})$. In addition,
$h$ is doubling, therefore it follows that we can bound the first
factor to obtain that
\begin{equation*}
    D^h_k \gtrsim \frac{6^{-k}n^2_{k-1}}{\g^{2}(2n_{k-2})\Delta(h,\h)\left(\frac{1}{\log(n_k)\g(2n_{k-1})}\right){\left(\g^{-1}(\log(n_k))+2n_{k-1}\right)^2}},
\end{equation*}
since, by hypothesis $(B)$, $n_k$ satisfies $\log(n_{k-1})\le
\g(n_{k-2})$ and $\g$ satisfies \eqref{eq:logg-1}. Now, again
by hypothesis $(B)$,
\begin{equation*}
D^h_k \ge
\frac{1}{6^k\g^{2}(2n_{k-2})}\frac{1}{\Delta(h,\h)\left(\frac{1}{\log(n_{k})\g(2n_{k-1})}\right)}.
\end{equation*}
Thus, if
\[
\varliminf_k\frac{1}{6^k\g^{2}(2n_{k-2})}\frac{1}{\Delta(h,\h)\left(\frac{1}{\log(n_{k})\g(2n_{k-1})}\right)}
>0,
\]
then $\HH{h}(E)>0$ and therefore $\HH{h}(B_\g)>0$.
\end{proof}

The following example not only illustrates this last result, but also will be crucial in the proof of Lemma \ref{lem:fiber}.
\begin{example}\label{ex:loglow}
Define $\g_r(x)=e^{x^\frac{2}{r}}$ for $r>0$ and consider the set $B_{\g_r}$.  Then $\h_r(x)=\frac{1}{\log^r(\frac{1}{x})}$ will be an expected lower bound for the dimension function for the set $B_{\g_r}$. Consider the family $h_\theta(x)=\frac{1}{\log^\theta(\frac{1}{x})}$ ($0<\theta<r$), which satisfy $h_\theta\prec \h_r$. In this context, $\frac{\h_r(x)}{h_\theta(x)} = \log^{\theta-r}(\frac{1}{x})$. Define the sequence $n_k$ as follows:
\[
    n_k=e^{kn^\frac{2}{r}_{k-1}}.
\]
Clearly the sequence is admissible, since
\begin{itemize}
\item[(A)] $n_{k}\ge 3\g_r(2n_{k-1})$, and
\item[(B)] $\log(n_k)\le \g_r(n_{k-1})$.
\end{itemize}
Then, for the quantity $D_k^h$ defined in Proposition \ref{pro:Bglow}, we have the following estimate: 
\begin{equation*}
D_k^h  \gtrsim
\frac{(\log\log(n_k)+n^\frac{2}{r}_{k-1})^{r-\theta}}{6^ke^{Mn^\frac{2}{r}_{k-2}}}
\ge
\frac{n_{k-1}^{2\frac{r-\theta}{r}}}{6^ke^{Mn^\frac{2}{r}_{k-2}}}, \qquad M=2^{\frac{2}{r}+1}.
\end{equation*}
Finally, for any $\e>0$ and  $M>0$, $n_k$ satisfies, for large
$k$,
\[
    \frac{n^\e_{k-1}}{6^ke^{Mn^\frac{2}{r}_{k-2}}}=
    \frac{e^{\e{kn_{k-2}^\frac{2}{r}}}}{6^ke^{Mn_{k-2}^\frac{2}{r}}}
    =\frac{e^{(\e k -M)n_{k-2}^\frac{2}{r}}}{6^k}\ge 1,
\]
so we conclude that $\varliminf_k D_k^h>0$. Therefore the set $E\sub B_{\g_r}$ constructed in the proof above satisfies $\HH{h_\theta}(E)>0$ for all $\theta<r$.
\end{example}

\subsection{Another Jarn\'ik type Theorem and proof of Lemma \ref{lem:fiber}}

For the proof of Lemma \ref{lem:fiber} we will need a different but essentially equivalent formulation of Jarn\'ik's theorem. We first recall that in the proof of Theorem \ref{thm:logFalpha} we  used the following theorem.
\begin{theorem}\label{thm:jarnikA}
 Let $\n=\{n_j\}_j$ be an increasing sequence with $n_{j+1}\ge n_j^j$ for all $j\in\N$. For $0<\alpha\le1$, if $A^\n_\alpha$ is defined as
\[
A^\n_\alpha=\left\{x\in[0,1]\setminus \Q: \forall j\ \exists\
p,q\ ; q\le n_j^\alpha; |x-\frac{p}{q}|<\frac{1}{n_j^2}\right\},
\]
then $\dim_H(A^\n_\alpha)=\alpha$.
\end{theorem}\index{Aa@$A^\n_\alpha$}

For the proofs of Theorem \ref{thm:jarnikA} and
Theorem \ref{thm:jarnikB}, we refer the reader to \cite{jar31},
\cite{bes34}, \cite{egg52}, \cite{fal86}, and \cite{fal03}.

Now we want to relate the sets $A^\n_\alpha$ and $B_\beta$ and their generalized versions. It is clear that for any $\alpha\in (0,1]$, we have the inclusion
$A^\n_\alpha\subset B_\frac{2}{\alpha}$. For $\alpha\in (0,1]$, if $x\in A^\n_\alpha$ then for each $j\in\N$ there exists a rational $\frac{p_j}{q_j}$ with $q_j\le n^\alpha_j$ such that $|x-\frac{p_j}{q_j}|<n_j^{-2}$, which is equivalent to $|xq_j-p_j|<q_jn_j^{-2}$. Therefore $|xq_j-p_j|\le q_j^{1-\frac{2}{\alpha}}$. Observe that if there were only finite values of $q$  for a given $x$, then $x$ has to be rational. For if $q_j=q_{j_0}$ for all $j\ge j_0$, then $|x-\frac{p_j}{q_{j_0}}|\to 0$ and this implies that $x\in \Q$. We conclude then that, for any $x\in A^\n_\alpha$, $d(xq,\mathbb Z)<\frac{1}{q^{\frac{2}{\alpha}-1}}$ for infinite many $q$ and therefore $x\in B_{\frac{2}{\alpha}}$. However, since the dimension of $A^\n_\alpha$ coincides with the one of
$B_\frac{2}{\alpha}$, one can expect that both sets have approximately comparable sizes.

We introduce the following  definition, which is the extended version of the definition of the  set $A^\n_\alpha$ in Theorem \ref{thm:jarnikA}.

\begin{definition}\label{def:jarnikh}
Let $\n=\{n_j\}_j$ be any increasing nonnegative sequence of
integers. Let $\f$ be an increasing function  defined on $\R_+$.
Define the set
\begin{equation*}
A^\n_\f:=\left\{x\in[0,1]\setminus \Q: \forall j\ \exists\ p,q\ ;
q\le \f(n_j); |x-\frac{p}{q}|<\frac{1}{n_j^2}\right\}.
\end{equation*} \index{Af@$A^\n_\alpha$}
\end{definition}
The preceding observation about the inclusion $A_\alpha\subset
B_\beta$ can be extended to this general setting. For a given
$\g$ as in the definition of $B_\g$, define
$\Gamma_\g(x)=\g^{-1}(x^2)$. Then the same calculations show that
$A^\n_{\Gamma_\g} \subset B_\g$.

We will need a converse relation between those sets, since we
want to prove a lower bound for the sets $A^\n_\f$ from the
estimates provided in Proposition \ref{pro:Bglow}.

\begin{lemma}\label{lem:EksubAf}
Let $\g$ and $B_\g$ be as in Proposition \ref{pro:Bglow}. Define
$\Gamma_\g(x)=\g^{-1}(x^2)$. Then, if $\m=\{m_k\}$ is the
sequence defining the set $E$ in the proof of
Proposition \ref{pro:Bglow}, then the set $E$ is contained in
$A^\n_{2\Gamma_\g}$, where $\n=\{n_k\}=\{\g(m_k)^\frac{1}{2}\}$.
\end{lemma}
\begin{proof}
 Recall that in the proof of Proposition \ref{pro:Bglow} we define the sets $G'_q$ as a union of intervals of the form $I=\left(\frac{r}{q}-\frac{1}{\g(q)};\frac{r}{q}+\frac{1}{\g(q)}\right)\setminus \Q$. The sets $H_n$ were defined as $   H_n:=\bigcup_{p\in\mathcal{P}_n^{2n}}G'_p$, where $\mathcal{P}_n^{2n}$ is the set of primes between $n$ and $2n$. We can therefore write
 \begin{equation*}
   H_n:=\bigcup I^n_j.
 \end{equation*}
Now, given a sequence $\m=\{m_k\}$, for each $k$, the set $E_k$
is defined as the union of all the intervals of $H_{m_k}$ that
belong to $E_{k-1}$, where $E_0=[0,1]$. If $E=\bigcap E_k$, any
$x\in E$ is in $E_k$ and therefore in some of the $I_j^{m_k}$. It
follows that there exists integers $r$ and $q$, $q\le 2m_k$ such
that
\begin{equation*}
 \left|x-\frac{r}{q}\right|<\frac{1}{\g(q)}<\frac{1}{\g(m_k)}=\frac{1}{n_k^2},\qquad q\le 2\g^{-1}(n_k^2).
\end{equation*}
Therefore $E\subset A_{2\Gamma_\g}^\n$.
\end{proof}

We remark that the above inclusion implies that any lower estimate on the size of $E$ would also be a lower estimate for
$A_{2\Gamma_\g}^\n$. 

We now conclude the proof of Theorem \ref{thm:sqrth3/2} by proving Lemma \ref{lem:fiber}: 

\begin{proof}[Proof of Lemma \ref{lem:fiber}:] 
Let $\h(x)=\frac{1}{\log(\frac{1}{x})}$.  For $r>1$, consider the function $\g_r$, the sequence $\m=\{m_k\}$ and the set $E_r$ as in Example \ref{ex:loglow}. Define $\f=2\Gamma_{\g_r}$, $\n$ and  $A^\n_{\f}$ as in Lemma \ref{lem:EksubAf}. It follows that $\f(x)=2\log(x^2)^\frac{r}{2}$, $n_j=e^{\frac{1}{2}n_{j-1}^{\frac{4}{r}j}}$ and 
\begin{equation*}
A^\n_\f:=\left\{x\in[0,1]\setminus \Q: \forall j\ \exists\ p,q\ ;
q\le \f(n_j); |x-\frac{p}{q}|<\frac{1}{n_j^2}\right\}.
\end{equation*}
Note that Lemma \ref{lem:EksubAf} says that the inclusion $E\subset A^\n_{2\Gamma_\g}$ \emph{always} holds, for any defining function $\g$, where $E$ is the \emph{substantial portion} of the set $B_\g$ (see Proposition \ref{pro:Bglow}). But we need the positivity of $\HH{\h}(E)$ to conclude that the set $A^\n_{2\Gamma_\g}$ also has positive $\HH{\h}$-measure. For the precise choices of $\h$ and $\g$, we obtain this last property from Example \ref{ex:loglow}. Precisely, $\HH{\h}(E)>0$ and therefore the set $A^\n_\f$ has positive $\HH{\h}$-measure. This concludes the proof of Lemma \ref{lem:fiber} and therefore the set constructed in the proof of Theorem \ref{thm:sqrth3/2} fulfills the condition of being an $F_\h$-set.
\end{proof}

As a final remark, we mention that it would be interesting to obtain sharp estimates in terms of generalized Hausdorff measures for any $F_\h$ class, not only for $F_\alpha$, in the spirit of Theorem \ref{thm:logFalpha}. The intuition here says that the reasonable dimension function for an upper bound for the class $F_\h$ is $\sqrt{\cdot}\h^{\frac{3}{2}} $. Therefore, a nice problem would be to construct an $F_\h$ set of zero $\HH{h}$-measure for a dimension function $h$ which is ``very close'' to $\h$. And, further, to estimate (as in Theorem \ref{thm:logFalpha}) how close to $\h$ such an $h$ can be in order to still allow us to perform the construction.

\section*{Acknowledgments}
We would like to thank to Pablo Shmerkin for many interesting
suggestions and remarks about this problem.

\bibliographystyle{amsalpha}

 \end{document}